\documentclass[11pt]{article}
\usepackage{amsmath,amsthm,amssymb}
\usepackage[usenames,dvipsnames]{xcolor}
\usepackage{enumerate}
\usepackage{graphicx}
\usepackage{cite}
\usepackage{comment}
\usepackage{oands}
\usepackage{tikz}
\usepackage{changepage}
\usepackage{bbm}
\usepackage{mathtools}
\usepackage[margin=1in]{geometry}
\usepackage[pagewise,mathlines]{lineno}
\usepackage{appendix}
\usepackage{stmaryrd} 
\usepackage{multicol}
\usepackage{microtype}
\usepackage[colorinlistoftodos]{todonotes}
\usepackage{enumitem}

\usepackage[pdftitle={Local metrics of the Gaussian free field},
  pdfauthor={Ewain Gwynne and Jason Miller},
colorlinks=true,linkcolor=NavyBlue,urlcolor=RoyalBlue,citecolor=PineGreen,bookmarks=true,bookmarksopen=true,bookmarksopenlevel=2,unicode=true,linktocpage]{hyperref}

\theoremstyle{plain}
\newtheorem{thm}{Theorem}[section]
\newtheorem{cor}[thm]{Corollary}
\newtheorem{lem}[thm]{Lemma}

\def\@rst #1 #2other{#1}
\newcommand\MR[1]{\relax\ifhmode\unskip\spacefactor3000 \space\fi
  \MRhref{\expandafter\@rst #1 other}{#1}}
\newcommand{\MRhref}[2]{\href{http://www.ams.org/mathscinet-getitem?mr=#1}{MR#2}}

\theoremstyle{definition}
\newtheorem{defn}[thm]{Definition}

\numberwithin{equation}{section}

\newcommand{\dsb}{\begin{adjustwidth}{2.5em}{0pt}
\begin{footnotesize}}
\newcommand{\dse}{\end{footnotesize}
\end{adjustwidth}}

\newcommand{\ssb}{\begin{adjustwidth}{2.5em}{0pt}}
\newcommand{\sse}{\end{adjustwidth}}

\newcommand{\aryb}{\begin{eqnarray*}}
\newcommand{\arye}{\end{eqnarray*}}
\def\alb#1\ale{\begin{align*}#1\end{align*}}
\def\allb#1\alle{\begin{align}#1\end{align}}
\newcommand{\eqb}{\begin{equation}}
\newcommand{\eqe}{\end{equation}}
\newcommand{\eqbn}{\begin{equation*}}
\newcommand{\eqen}{\end{equation*}}

\newcommand{\BB}{\mathbbm}
\newcommand{\ol}{\overline}

\newcommand{\op}{\operatorname}

\newcommand{\frk}{\mathfrak}
\newcommand{\eqD}{\overset{d}{=}}
\newcommand{\ep}{\varepsilon}
\newcommand{\rta}{\rightarrow}

\newcommand{\wt}{\widetilde}
 
\newcommand{\mcl}{\mathcal}

\newcommand{\bdy}{\partial}
\newcommand{\rng}{\mathring}

\let\originalleft\left
\let\originalright\right
\renewcommand{\left}{\mathopen{}\mathclose\bgroup\originalleft}
\renewcommand{\right}{\aftergroup\egroup\originalright}

\title{Local metrics of the Gaussian free field}

\author{Ewain Gwynne and Jason Miller}
\date{ }
\author{Ewain Gwynne and Jason Miller  \\ {\it University of Cambridge}}

\begin{document}

\maketitle 

\begin{abstract}
We introduce the concept of a \emph{local metric} of the Gaussian free field (GFF) $h$, which is a random metric coupled with $h$ in such a way that it depends locally on $h$ in a certain sense. This definition is a metric analog of the concept of a local set for $h$. We establish general criteria for two local metrics of the same GFF $h$ to be bi-Lipschitz equivalent to each other and for a local metric to be a.s.\ determined by $h$. 
Our results are used in subsequent works which prove the existence, uniqueness, and basic properties of the $\gamma$-Liouville quantum gravity (LQG) metric for all $\gamma \in (0,2)$, but no knowledge of LQG is needed to understand this paper.  
\end{abstract}


\tableofcontents

\section{Introduction}
\label{sec-intro}

\subsection{Overview}
\label{sec-overview}
  
Let $U\subset \BB C$ be an open planar domain.  The Gaussian free field (GFF) $h$ on $U$ is a random distribution (generalized function) on $U$ which can be thought of as a generalization of Brownian motion but with two time variables instead of one, in the sense that the one-dimensional Gaussian free field is simply the Brownian bridge.  We refer to Section~\ref{sec-gff-prelim},~\cite{shef-gff}, and/or the introductory sections of~\cite{ss-contour,shef-zipper,ig1,ig4} for more on the GFF. 

If $(h,A)$ is a coupling of $h$ with a random compact subset of $U$, we say that $A$ is a \emph{local set} for $h$ if for any open set $V\subset U$, the event $\{A\subset V\}$ is conditionally independent from $h|_{U\setminus V}$ given $h|_V$.\footnote{The restriction of $h$ to an open set $V$ can be defined as the restriction of the distributional pairing $\phi \mapsto (h,\phi)$ to test functions $\phi$ which are supported on $V$. The $\sigma$-algebra generated by $h|_K$ for a closed set $K$ can be defined as $\bigcap_{\ep > 0} \sigma(h|_{B_\ep(K)})$, where $B_\ep(K)$ is the Euclidean $\ep$-neighborhood of $K$. Hence it makes sense to speak of ``conditioning on $h|_K$" or to say that a random variable is ``determined by $h|_K$".}
In other words, $A$ depends ``locally" on $h$, although $A$ is not required to be determined by $h$. 
Local sets of $h$ were first defined in~\cite[Lemma 3.9]{ss-contour}.  
Important examples of local sets include sets which are independent from $h$ as well as so-called ``level lines"~\cite{ss-contour} and ``flow lines"~\cite{ig1,ig2,ig3,ig4} of $h$, both of which are SLE$_\kappa$-type curves that are a.s.\ determined by $h$.

In this paper, we will study random metrics coupled with the GFF instead of random sets. 
As we will explain in more detail Section~\ref{sec-lqg}, this work is motivated by the question of constructing the Liouville quantum gravity metric for $\gamma \in (0,2)$. 
This is the distance function associated with the Riemannian metric tensor ``$e^{\gamma h} \,(dx^2  +dy^2)$", where $dx^2 + dy^2$ denotes the Euclidean metric tensor on $U$, which is in some sense a canonical model of a random two-dimensional Riemannian metric tensor. However, the ideas we develop here will apply in a more general framework.

We will now introduce a concept of a \emph{local metric} of the GFF, which is directly analogous to the above definition of a local set.
We first need some preliminary definitions.
Suppose $(X,D)$ is a metric space. 
\medskip 

\noindent
For a curve $P : [a,b] \rta X$, the \emph{$D$-length} of $P$ is defined by 
\eqbn
\op{len}\left( P ; D  \right) := \sup_{T} \sum_{i=1}^{\# T} D(P(t_i) , P(t_{i-1})) 
\eqen
where the supremum is over all partitions $T : a= t_0 < \dots < t_{\# T} = b$ of $[a,b]$. Note that the $D$-length of a curve may be infinite.
\medskip

\noindent
For $Y\subset X$, the \emph{internal metric of $D$ on $Y$} is defined by
\eqb \label{eqn-internal-def}
D(x,y ; Y)  := \inf_{P \subset Y} \op{len}\left(P ; D \right) ,\quad \forall x,y\in Y 
\eqe 
where the infimum is over all paths $P$ in $Y$ from $x$ to $y$. 
Then $D(\cdot,\cdot ; Y)$ is a metric on $Y$, except that it is allowed to take infinite values.  
\medskip
 
\noindent
We say that $(X,D)$ is a \emph{length space} if for each $x,y\in X$ and each $\ep > 0$, there exists a curve of $D$-length at most $D(x,y) + \ep$ from $x$ to $y$. 
\medskip

\noindent
A \emph{continuous metric} on an open domain $U\subset\BB C$ is a metric $D$ on $U$ which induces the Euclidean topology on $U$, i.e., the identity map $(U,|\cdot|) \rta (U,D)$ is a homeomorphism. 
We equip the space of continuous metrics on $U$ with the local uniform topology for functions from $U\times U$ to $[0,\infty)$ and the associated Borel $\sigma$-algebra.
Note that the space of continuous metrics is not complete w.r.t.\ this topology.  
We allow a continuous metric to satisfy $D(u,v) = \infty$ if $u$ and $v$ are in different connected components of $U$.
In this case, in order to have $D^n\rta D$ w.r.t.\ the local uniform topology we require that for large enough $n$, $D^n(u,v) = \infty$ if and only if $D(u,v)=\infty$.

\begin{lem}
Let $D$ be a continuous length metric on $U$ and let $V\subset U$ be open. 
The internal metric $D(\cdot,\cdot ; V)$ is a continuous length metric on $V$.
\end{lem}
\begin{proof}
Since $D(\cdot,\cdot ; V) \geq D|_V$, it is clear that the identity map $(V , D(\cdot,\cdot ; V)) \rta (V , |\cdot|)$ is continuous. 
To check that the inverse map is continuous, suppose $\{z_n\}_{n\in\BB N}$ is a sequence in $V$ which converges to $z\in V$ with respect to the Euclidean topology. 
By the continuity of $D$, for large enough $n\in\BB N$, the $D$-distance from $z_n$ to $z$ is smaller than the $D$-distance from $z$ to $\bdy V$. 
This implies that a path of near-minimal $D$-length from $z_n$ to $z$ must stay in $V$, so since $D$ is a length metric we have $D(z,z_n) = D(z,z_n; V)$ for large enough $n$.
By the continuity of $D$, we have $D(z,z_n) \rta 0$ so also $D(z,z_n; V) \rta 0$. 
\end{proof}
\medskip
 
\begin{defn}[Local metric] \label{def-local-metric}
Let $U\subset \BB C$ be a connected open set and let $(h,D)$ be a coupling of a GFF on $U$ and a random continuous length metric on $U$.  We say that $D$ is a \emph{local metric} for $h$ if for any open set $V\subset U$, the internal metric $D(\cdot,\cdot;  V)$ is conditionally independent from the pair $(h|_{U\setminus V} , D(\cdot,\cdot; U\setminus \ol V))$ given $h|_V$.
\end{defn}

By convention, we define $D(\cdot,\cdot; U\setminus \ol V)$ to be a graveyard point in the probability space if $U\setminus \ol V = \emptyset$. 
We emphasize that in Definition~\ref{def-local-metric}, $D(\cdot,\cdot;  V)$ is required to be conditionally independent from the \emph{pair} $(h|_{U\setminus V} , D(\cdot,\cdot; U\setminus \ol V))$ given $h|_V$. 
This means that, unlike in the case of a local set, a random metric $D$ which is independent from $h$ is not necessarily a local metric for $h$.
If $D$ is determined by $h$, then $D$ is a local metric for $h$ if and only if $D(\cdot,\cdot ; V)$ is determined by $h|_V$ for each open set $V\subset U$. 
If $D$ is not necessarily determined by $h$, then Definition~\ref{def-jointly-local} implies that $D(\cdot,\cdot;V)$ is conditionally independent from $D(\cdot,\cdot;U\setminus \ol V)$ given $h$; see also Lemma~\ref{lem-open-ind} for a stronger version of this statement.
See Section~\ref{sec-local-metric-properties} for some equivalent formulations of Definition~\ref{def-local-metric}.

The goal of this paper is to prove several general theorems about local metrics of the GFF.
We will give a condition under which two local metrics are bi-Lipschitz equivalent to each other (Theorem~\ref{thm-bilip}) and a condition under which a local metric is a measurable function of the field (Theorem~\ref{thm-msrble-general}). 

As mentioned above, our results play an important role in related works~\cite{dddf-lfpp,lqg-metric-estimates,gm-confluence,gm-uniqueness,gm-coord-change} which construct a certain special family of local metrics of the GFF: the \emph{$\gamma$-Liouville quantum gravity} (LQG) metric for $\gamma \in (0,2)$. We discuss these works further in Section~\ref{sec-lqg}. However, we emphasize that one does not need to know anything about LQG to understand this paper.   
\bigskip

\noindent\textbf{Acknowledgments.} We thank an anonymous referee for helpful comments on an earlier version of this article. We thank Jian Ding, Julien Dub\'edat, Alex Dunlap, Hugo Falconet, Josh Pfeffer, Scott Sheffield, and Xin Sun for helpful discussions.  
EG was supported by a Herchel Smith fellowship and a Trinity College junior research fellowship.
JM was supported by ERC Starting Grant 804166. 
 
\subsection{Main results}
\label{sec-local-metrics}
  
Before stating our main results concerning local metrics, we need some additional definitions which build on Definition~\ref{def-local-metric}. 
 
\begin{defn}[Jointly local metrics] \label{def-jointly-local}
Let $U\subset \BB C$ be a connected open set and let $(h,D_1,\dots,D_n)$ be a coupling of a GFF on $U$ and $n$ random continuous length metrics.
We say that $D_1,\dots,D_n$ are \emph{jointly local metrics} for $h$ if for any open set $V\subset U$, the collection of internal metrics $\{ D_j(\cdot,\cdot;  V) \}_{j = 1,\dots,n}$ is conditionally independent from $(h|_{U\setminus V} ,  \{ D_j(\cdot,\cdot;  U\setminus \ol V) \}_{j = 1,\dots,n}   )$ given $h|_V$.
\end{defn}

The following lemma gives a convenient way to produce jointly local metrics. See~\cite[Lemma~3.10]{ss-contour} for the analog of the lemma for local sets.

\begin{lem} \label{lem-jointly-local}
Let $(h,D_1,\dots,D_n)$ be a coupling of a GFF on a domain $U\subset\BB C$ with $n$ random continuous length metrics such that each $D_j$ for $j=1,\dots,n$ is local for $h$ and $(D_1,\dots,D_n)$ are conditionally independent given $h$. 
Then $ D_1,\dots,D_n $ are jointly local for $h$.
\end{lem}
\begin{proof}
Fix $V\subset U$. We first treat the case when $n=2$.  
We will apply the following elementary probability fact: if $(X,Y,Z)$ is a coupling of three random variables such that $X$ is independent from $Y$, $X$ is independent from $Z$, and $(Y,Z)$ are conditionally independent given $X$, then $X$ is independent from $(Y,Z)$. See, e.g.,~\cite[Lemma 3.5]{ss-contour}. 
For our purposes, we will take 
\eqb
X = \left( D_1(\cdot,\cdot; U\setminus \ol V) , D_2(\cdot,\cdot ; U\setminus \ol V) ,  h|_{U\setminus\ol V} \right) ,\quad
Y = D_1(\cdot,\cdot;V), \quad
Z = D_2(\cdot,\cdot;V)  
\eqe
and apply the statement to the conditional law of $(X,Y,Z)$ given $h|_V$. 

Since $D_1$ and $D_2$ are conditionally independent given $h$, the conditional law of $ D_1(\cdot,\cdot;V)$ given $\left( D_1(\cdot,\cdot; U\setminus \ol V) , D_2(\cdot,\cdot ; U\setminus \ol V) ,  h \right)$ is the same as the conditional law of $D_1(\cdot,\cdot ; V)$ given only $\left( D_1(\cdot,\cdot; U\setminus \ol V)  ,  h \right)$.
By the locality of $D_1$, this is the same as the conditional law of $D_1(\cdot,\cdot ; V)$ given only $h|_V$. 
Hence, in the notation above, $X$ and $Y$ are conditionally independent given $h|_V$.
Similarly, $X$ and $Z$ are conditionally independent given $h|_V$. 
Since $D_1$ and $D_2$ are conditionally independent given $h$, it follows that also $Y$ and $Z$ are conditionally independent given $X$ and $h|_V$.
The above probability fact therefore shows that $(Y,Z)$ is conditionally independent from $X$ given $h|_V$.
This means precisely that $D_1$ and $D_2$ are jointly local for $h$.

This completes the proof when $n=2$. The case when $n\geq 3$ follows from induction and a similar argument to the one above. 
\end{proof}

We will sometimes work with GFF's which are naturally defined modulo additive constant. 
When we do so, we will typically normalize the field so that its circle average $h_r(z)$ over $\bdy B_r(z)$ is zero for some $r >0$ and $z\in\BB C$ (see~\cite[Section 3.1]{shef-kpz} for the definition and basic properties of circle averages). 
We will be interested in local metrics which behave nicely when we make a different choice of normalization. 

\begin{defn}[Additive local metrics] \label{def-additive-local}
Let $U\subset \BB C$ be a connected open set and let $(h,D_1,\dots,D_n)$ be a coupling of a GFF on $U$ and $n$ random continuous length metric which are jointly local for $h$.
For $\xi \in\BB R$, we say that $(D_1,\dots,D_n)$ are \emph{$\xi$-additive} for $h$ if for each $z\in U$ and each $r> 0$ such that $B_r(z) \subset U$, the metrics $(e^{-\xi h_r(z)} D_1,\dots, e^{-\xi h_r(z)} D_n)$ are jointly local metrics for $h - h_r(z)$. 
\end{defn}

The first main result of this article is the following criterion for two local metrics to be bi-Lipschitz equivalent.  Roughly speaking, it states that if we can compare the distance across an annulus for one metric to the diameter of a circle w.r.t.\ the internal metric on an annulus for the other metric with high probability at all scales, then we get an a.s.\ global comparison of the metrics. 

\begin{thm}[Bi-Lipschitz equivalence of local metrics] \label{thm-bilip}
Let $\xi \in\BB R$, let $h$ be a whole-plane GFF normalized so that $h_1(0) = 0$, let $U\subset\BB C$, and let $(h,D,\wt D)$ be a coupling of $h$ with two random continuous metrics on $U$ which are jointly local and $\xi$-additive for $h|_U$.  
There is a universal constant $p \in (0,1)$ such that the following is true.
Suppose there is a deterministic constant $C>0$ such that for each compact set $K\subset U$, there exists $r_K > 0$ such that
\eqb \label{eqn-bilip}
\BB P\left[  \sup_{u,v \in \bdy B_r(z)} \wt D\left(u,v; B_{2r}(z) \setminus \ol{B_{r/2}(z)} \right) \leq C D(\bdy B_{r/2}(z) , \bdy B_r(z) ) \right] \geq p    ,\quad \forall z\in K, \quad \forall r \in (0,r_K] .
\eqe 
Then a.s.\ $\wt D(z,w) \leq C D(z,w)$ for each $z,w\in U$. 
\end{thm}

We also have a criterion for a local metric to be determined by $h$, which says that, roughly speaking, if a local metric is determined by $h$ up to bi-Lipschitz equivalence then it is in fact itself determined by $h$.

\begin{thm}[Measurability of local metrics] \label{thm-msrble-general}
Let $U\subset\BB C$, let $h$ be a GFF on $U$, and let $(h,D )$ be a coupling of $h$ with a random continuous length metric which is local for $h $.
Assume that $D$ is determined by $h$ up to bi-Lipschitz equivalence in the following sense. Suppose we condition on $h$ and let $D,\wt D$ be conditionally i.i.d.\ samples from the conditional law of $D$ given $h$. 
There is a random constant $C = C_h >1$, depending only $h$, such that a.s.\ $ \wt D (z,w) \leq C D (z,w)$ for each $z,w\in U$. 
Then $D$ is a.s.\ determined by $h $, i.e., one can take $C = 1$. 
\end{thm}

Combining Theorem~\ref{thm-bilip} with Theorem~\ref{thm-msrble-general} yields the following corollary.

\begin{cor} \label{cor-bilip-msrble}
There is a universal constant $p\in (0,1)$ such that the following is true. Let $U\subset\BB C$ be a domain which contains the unit disk, let $h$ be a whole-plane GFF normalized so that $h_1(0) = 0$, and let $(h,D )$ be a coupling of $h$ with a random continuous length metric on $U$ which is local for $h|_U$ and satisfies the following hypotheses.
\begin{enumerate}
\item $D$ is $\xi$-additive for $h|_U$ for some $\xi\in\BB R$ (Definition~\ref{def-additive-local}). 
\item Condition on $h$ and let $D$ and $\wt D$ be conditionally i.i.d.\ samples from the conditional law of $D$ given $h$. There is a deterministic constant $C > 0$ such that for each compact set $K\subset U$, there exists $r_K > 0$ such that~\eqref{eqn-bilip} holds for this choice of $D$ and $\wt D$.  
\end{enumerate}  
Then $D$ is a.s.\ determined by $h$. 
\end{cor}
\begin{proof}
We first claim that $(D,\wt D)$ is a pair of $\xi$-additive local metrics for $h|_U$. 
We know from Definition~\ref{def-additive-local} that for each $z\in U$ and each $r > 0$ such that $B_r(z)\subset U$, each of $e^{-\xi h_r(z)} D$ and $e^{-\xi h_r(z)} \wt D$ is individually local for $h|_U - h_r(z)$. 
Since $B_1(0) , B_r(z) \subset U$, it follows that $h|_U$ and $h|_U - h_r(z)$ determine each other. 
Since $e^{-\xi h_r(z)} D$ and $e^{-\xi h_r(z)} \wt D$ are conditionally independent given $h|_U$, they are also conditionally independent given $h|_U - h_r(z)$. By Lemma~\ref{lem-jointly-local}, these two metrics are jointly local for $h|_U -h_r(z)$. 
Theorem~\ref{thm-bilip} tells us that if~\eqref{eqn-bilip} holds for a large enough universal $p\in (0,1)$, then a.s.\ $\wt D (z,w) \leq C D (z,w)$ for each $z,w\in U$. 
Therefore, Theorem~\ref{thm-msrble-general} implies that $D$ is a.s.\ determined by $h$. 
\end{proof}

\subsection{Applications to Liouville quantum gravity}
\label{sec-lqg}

Here we briefly summarize how our results are used in the construction of the $\gamma$-Liouville quantum gravity (LQG) metric for general $\gamma\in (0,2)$.
This section is included only for context and is not needed to understand the rest of the paper.

For $\gamma \in (0,2)$, a \emph{$\gamma$-LQG surface} is, heuristically speaking, the random two-dimensional Riemannian manifold parameterized by a domain $U\subset\BB C$ whose Riemannian metric tensor is $e^{\gamma h} (dx^2 + dy^2)$, where $h$ is some variant of the GFF on $U$ and $dx^2+dy^2$ is the Euclidean metric tensor.
This definition does not make literal sense since the GFF is only a distribution, not a function, so cannot be exponentiated. 
So, one needs to use regularization procedures to define LQG rigorously.
Previous work has constructed the volume form associated with an LQG surface, called the \emph{$\gamma$-LQG area measure}~\cite{kahane,shef-kpz,rhodes-vargas-review}.
This is a random measure $\mu_h$ which can be obtained as a limit of regularized versions of $e^{\gamma h}\,dz$, where $dz$ denotes Lebesgue measure. 

It is expected that a $\gamma$-LQG surface also admits a canonical metric. 
This metric should be a local metric for $h$ which is, in some sense, obtained by exponentiating $h$. Previously, such a metric was only constructed in the special case when $\gamma=\sqrt{8/3}$~\cite{lqg-tbm1,lqg-tbm2,lqg-tbm3}, in which case the associated metric space is isometric to the Brownian map~\cite{legall-uniqueness,miermont-brownian-map}.  
 
Ding, Dub\'edat, Dunlap, and Falconet~\cite{dddf-lfpp} showed that for general $\gamma \in (0,2)$, a certain natural approximation scheme for the $\gamma$-LQG metric called \emph{Liouville first passage percolation} (LFPP) admits non-trivial subsequential limiting metrics. To construct a metric on $\gamma$-LQG, one wants to show that there is a unique subsequential limit and that it satisfies certain scale invariance properties. This is accomplished in~\cite{gm-uniqueness}, building on~\cite{lqg-metric-estimates,gm-confluence} and the present paper. 

It can be checked that every subsequential limit of LFPP is a local metric of the GFF (see~\cite[Section 2]{lqg-metric-estimates}). 
Hence our results can be applied to study such subsequential limits. 
In particular, Corollary~\ref{cor-bilip-msrble} will be used in~\cite{lqg-metric-estimates} to show that every subsequential limit can be realized as a measurable function of the GFF. 
Theorem~\ref{thm-bilip} is used in~\cite{gm-uniqueness} to show that certain pairs of subsequential limiting metrics are bi-Lipschitz equivalent, which reduces the problem of proving uniqueness of the subsequential limit to the (quite involved) problem of showing that the two bi-Lipschitz equivalent metrics in fact differ by a scaling (i.e., the ratio of the two metrics is a positive and finite constant).
Theorem~\ref{thm-bilip} is also used in~\cite{gm-coord-change} as an intermediate step in the proof of the conformal covariance of the LQG metric.

\subsection{Outline}
\label{sec-outline}

The rest of this paper is structured as follows.
In Section~\ref{sec-gff-prelim}, we review some facts about the Gaussian free field and record some elementary properties of local metrics. 
In Section~\ref{sec-annulus-iterate} we prove a general lemma (Lemma~\ref{lem-annulus-iterate}) concerning the near-independence of events which depend on the GFF and a collection of jointly local metrics restricted to disjoint concentric annuli.
This lemma is an extension of a result from~\cite{mq-geodesics} and will also be used in~\cite{lqg-metric-estimates,gm-confluence,gm-uniqueness,gm-coord-change}. 
In Section~\ref{sec-bilip}, we use this general ``independence across annuli" lemma to prove Theorem~\ref{thm-bilip}.
In Section~\ref{sec-msrble}, we prove Theorem~\ref{thm-msrble-general}.

\section{Preliminaries}

\subsection{Basic notation}
\label{sec-notation}

\noindent
We write $\BB N = \{1,2,3,\dots\}$ and $\BB N_0 = \BB N \cup \{0\}$. 
\medskip

\noindent
For $a < b$, we define the discrete interval $[a,b]_{\BB Z}:= [a,b]\cap\BB Z$. 
\medskip
 
\noindent
If $f  :(0,\infty) \rta \BB R$ and $g : (0,\infty) \rta (0,\infty)$, we say that $f(\ep) = O_\ep(g(\ep))$ (resp.\ $f(\ep) = o_\ep(g(\ep))$) as $\ep\rta 0$ if $f(\ep)/g(\ep)$ remains bounded (resp.\ tends to zero) as $\ep\rta 0$. 
\medskip
 
\noindent
For $z\in\BB C$ and $r>0$, we write $B_r(z)$ for the Euclidean ball of radius $r$ centered at $z$. We also define the open annulus
\eqb \label{eqn-annulus-def}
\BB A_{r_1,r_2}(z) := B_{r_2}(z) \setminus \ol{B_{r_1}(z)} ,\quad\forall 0 < r_r < r_2 < \infty .
\eqe

\subsection{The Gaussian free field}
\label{sec-gff-prelim}
 
Here we give a brief review of the definition of the zero-boundary and whole-plane Gaussian free fields. We refer the reader to~\cite{shef-gff} and the introductory sections of~\cite{ss-contour,ig1,ig4} for more detailed expositions. 

For an open domain $U\subset \BB C$ with harmonically non-trivial boundary (i.e., Brownian motion started from a point in $U$ a.s.\ hits $\bdy U$), we define $\mcl H(U)$ be the Hilbert space completion of the set of smooth, compactly supported functions on $U$ with respect to the \emph{Dirichlet inner product},
\eqb \label{eqn-dirichlet}
(\phi,\psi)_\nabla = \frac{1}{2\pi} \int_U \nabla \phi(z) \cdot \nabla \psi(z) \,dz .
\eqe
In the case when $U= \BB C$, constant functions $c$ satisfy $(c,c)_\nabla = 0$, so to get a positive definite norm in this case we instead take $\mcl H(\BB C)$ to be the Hilbert space completion of the set of smooth, compactly supported functions $\phi$ on $\BB C$ with $\int_{\BB C} \phi(z) \,dz = 0$, with respect to the same inner product~\eqref{eqn-dirichlet}.  
 
The \emph{(zero-boundary) Gaussian free field} on $U$ is defined by the formal sum
\eqb \label{eqn-gff-sum}
h = \sum_{j=1}^\infty X_j \phi_j 
\eqe
where the $X_j$'s are i.i.d.\ standard Gaussian random variables and the $\phi_j$'s are an orthonormal basis for $\mcl H(U)$. The sum~\eqref{eqn-gff-sum} does not converge pointwise, but it is easy to see that for each fixed $\phi \in \mcl H(U)$, the formal inner product $(h ,\phi)_\nabla$ is a mean-zero Gaussian random variable and these random variables have covariances $\BB E [(h,\phi)_\nabla (h,\psi)_\nabla] = (\phi,\psi)_\nabla$. In the case when $U \not=\BB C$ and $U$ has harmonically non-trivial boundary, one can use integration by parts to define the ordinary $L^2$ inner products $(h,\phi) := -2\pi (h,\Delta^{-1}\phi)_\nabla$, where $\Delta^{-1}$ is the inverse Laplacian with zero boundary conditions, whenever $\Delta^{-1} \phi \in \mcl H(U)$. 

In the case when $U=\BB C$, one can similarly define $(h ,\phi) := -2\pi (h ,\Delta^{-1}\phi)_\nabla$ where $\Delta^{-1}$ is the inverse Laplacian normalized so that $\int_{\BB C} \Delta^{-1} \phi(z) \, dz = 0$. With this definition, one has $(h+c , \phi) = (h ,\phi) + (c,\phi) = (h,\phi)$ for each $\phi \in \mcl H(\BB C)$, so the whole-plane GFF is only defined modulo a global additive constant. We will typically fix this additive constant by requiring that the circle average $h_r(z)$ over $\bdy B_r(z)$ is zero for some $z\in\BB C$ and $r > 0$. That is, we consider the field $h - h_r(z)$, which is well-defined not just modulo additive constant. We refer to~\cite[Section 3.1]{shef-kpz} for more on the circle average. 
The law of the whole-plane GFF is scale and translation invariant modulo additive constant, which means that for $z\in\BB C$ and $r>0$ one has $h(r\cdot +z) - h_r(z) \eqD h - h_1(0)$. 
 
The zero-boundary GFF on a domain $U$ with harmonically non-trivial boundary possesses the following Markov property (see, e.g.,~\cite[Section 2.6]{shef-gff}).  Let $V\subset U$ be a sub-domain with harmonically non-trivial boundary.
Then we can write $h = \frk h + \rng h$, where $\frk h$ is a random distribution on $U$ which is harmonic on $V$ and is determined by $h|_{U\setminus V}$; and $\rng h$ is a zero-boundary GFF on $V$ which is independent from $h|_{U\setminus V}$. 

In the whole-plane case, the Markov property is slightly more complicated due to the need to fix the additive constant. 
We will use the following version, which is proven in~\cite[Lemma 2.2]{gms-harmonic}.

\begin{lem}[\!\!\cite{gms-harmonic}] \label{lem-whole-plane-markov}
Let $h$ be a whole-plane GFF with the additive constant chosen so that $h_1(0) = 0$. 
For each open set $V\subset\BB C$ with harmonically non-trivial boundary, we have the decomposition
\eqb
h  =   \frk h + \rng h
\eqe
where $\frk h$ is a random distribution which is harmonic on $V$ and is determined by $h|_{\BB C\setminus V}$ and $\rng h$ is independent from $\frk h$ and has the law of a zero-boundary GFF on $V$ minus its average over $\bdy \BB D \cap V$. If $V$ is disjoint from $\bdy \BB D$, then $\rng h$ is a zero-boundary GFF and is independent from $h|_{\BB C\setminus V}$. 
\end{lem}

\subsection{Further basic properties of local metrics}
\label{sec-local-metric-properties}

Local metrics are related to local sets in the sense of~\cite[Lemma 3.9]{ss-contour} in the following manner. 

\begin{lem} \label{lem-ball-local}
Let $(h,D)$ be a coupling of a GFF on $U$ and a random continuous length metric on $U$. 
For $z\in U$ and $s >0$, let $\mcl B_s(z;D )$ be the open $D$-metric ball of radius $s$ centered at $z$. 
\begin{enumerate}
\item If $D$ is a local metric for $h$ and $\tau$ is a stopping time for the filtration generated by $(\ol{\mcl B_s(z;D )} , h|_{\ol{\mcl B_s(z;D )}})$, then $\ol{\mcl B_\tau(z;D )}$ is a local set for $h$. \label{item-ball-local}
\item If $D$ is determined by $h$ and each closed metric ball $\ol{\mcl B_s(z;D)}$ for $z\in U$ and $s >0$ is a local set for $h$, then $D$ is a local metric for $h$.  \label{item-ball-local'}
\end{enumerate}
\end{lem}

Assertion~\ref{item-ball-local'} is \emph{not} true without the assumption that $D$ is determined by $h$. 
A counterexample can be found by considering a random metric which is independent from $h$; see the discussion just after Definition~\ref{def-local-metric}. 

\begin{proof}[Proof of Lemma~\ref{lem-ball-local}]
\noindent\textit{Proof of Assertion~\ref{item-ball-local}.}
We first treat the case of a deterministic time $s> 0$. 
We will use the following criterion from~\cite[Lemma 3.9]{ss-contour}: a closed set $A$ coupled with $h$ is a local set if and only if for each open set $V\subset U$, the event $\{A\subset V\}$ is conditionally independent from $h|_{U \setminus V}$ given $h|_V$. 
Suppose now that we are given an open set $V\subset U$ and a deterministic $s>0$.  
The event $\{\ol{\mcl B_s(z;D)} \subset V\}$ is empty if $z\notin V$, and if $z \in V$ it is the same as the event that the $D$-distance from $z$ to each point of $\bdy V$ is strictly larger than $s$. 
This event is determined by the internal metric $D(\cdot,\cdot;V)$, so it is conditionally independent from $h|_{U\setminus V}$ given $h|_V$ by Definition~\ref{def-local-metric}. 
 
The case of stopping times which take on only countably many possible values is immediate from the case of deterministic times. The case of general stopping times follows from the standard strong Markov property argument (i.e., look at the times $2^{-n}  \lceil 2^n \tau \rceil$ and send $n\rta\infty$) and the fact that local sets behave nicely under limits~\cite[Lemma 6.8]{qle}. 
\medskip

\noindent\textit{Proof of Assertion~\ref{item-ball-local'}.} Assume that $D$ is determined by $h$ and let $V\subset U$ be open. Our locality assumption on metric balls together with~\cite[Lemma 3.9]{ss-contour} implies that for each $z\in V$ and $s > 0$, the $D$-metric ball $\ol{\mcl B_s(z;D)}$ is determined by $h|_V$ on the event $\{ \ol{\mcl B_s(z;D)} \subset V\} = \{D(z,\bdy V) > s\}$. Letting $z$ vary over $V \cap \BB Q^2$, letting $s$ vary over $(0,\infty)\cap \BB Q$, and using the continuity of $D$ shows that the set
\eqb
\left\{ (z,w) \in V\times V : D (z,w) < D (z,\bdy V) \right\} 
\eqe
and the restriction of $D$ to this set are each determined by $h|_V$.

Now suppose that $W\subset V$ is open and bounded with $\ol W\subset V$. 
By the continuity of $D$, a.s.\ $D(W,\bdy V) > 0$. 
By the conclusion of the preceding paragraph, $h|_V$ determines the set $\{(z,w) \in W\times W : D(z,w) < D(W,\bdy V)\}$ and the restriction of $D$ to this set. 
This information, in turn, determines $D(\cdot,\cdot;W)$. 
Letting $W$ increase to all of $V$ now shows that $h|_V$ determined $D(\cdot,\cdot;V)$. 
Since $D$ is assumed to be determined by $h$, this implies that $D$ is a local metric for $h$. 
\end{proof}

There are a few arbitrary choices in Definition~\ref{def-local-metric} involving whether to restrict $h$ to an open set or to its closure. The following lemma says that these choices do not matter. 

\begin{lem} \label{lem-local-equiv}
Let $(h,D)$ be a coupling of a GFF on $U$ and a random continuous length metric on $U$.
The following are equivalent.
\begin{enumerate}
\item $D$ is a local metric for $h$. \label{item-equiv-local}
\item (Replacing $h|_{U\setminus V}$ by $h$) For each open set $V\subset U$, the internal metric $D(\cdot,\cdot ; V)$ is conditionally independent from the pair $(h , D(\cdot,\cdot ; U\setminus \ol V))$ given $h|_{  V}$. \label{item-equiv-all}
\item (Conditioning on $h|_{\ol V}$ instead of $h|_V$) For each open set $V\subset U$, the internal metric $D(\cdot,\cdot ; V)$ is conditionally independent from the pair $(h , D(\cdot,\cdot ; U\setminus \ol V))$ given $h|_{\ol V}$. \label{item-equiv-closure}
\end{enumerate}
\end{lem}
\begin{proof}
Fix an open set $V\subset U$. 
Since $h$ is determined by $h|_V$ and $h|_{U\setminus V}$, it is obvious that~\ref{item-equiv-local} is equivalent to~\ref{item-equiv-all}. 
That~\ref{item-equiv-all} implies~\ref{item-equiv-closure} is a consequence of the following probability fact: if $X,Y,Z$ are random variables such that $X$ and $(Y,Z)$ are independent, then $X$ and $Z$ are conditionally independent given $Y$. Indeed, if we assume~\ref{item-equiv-all} then~\ref{item-equiv-closure} is immediate from this probability fact applied under the conditional law given $h|_V$ and with $X = D(\cdot,\cdot ; V)$, $Y = h|_{\ol V}$, and $Z =  (h , D(\cdot,\cdot ; U\setminus \ol V))$. 

Now assume that $(h,D)$ satisfies~\ref{item-equiv-closure}. 
For each open set $W\subset V$ with $\ol W\subset V$, we know that the metric $D(\cdot,\cdot ; W)$ is conditionally independent from the pair $(h , D(\cdot,\cdot ; U\setminus \ol W))$ given $h|_{\ol W}$.  
The field $h|_{\ol W}$ is determined by $h|_V$. 
The metric $D (\cdot,\cdot ; U\setminus \ol V)$ is equal to the internal metric of $D(\cdot,\cdot; U\setminus \ol W)$ on $U\setminus \ol V$, so is determined by $D(\cdot,\cdot ; U\setminus \ol W)$.  
Therefore, $D (\cdot,\cdot ; W)$ is conditionally independent from $(h  , D(\cdot,\cdot ; U\setminus \ol V))$ given $h|_V$. 
Letting $W$ increase to all of $V$ now shows that $D(\cdot,\cdot ; V)$ is conditionally independent from $(h , D(\cdot,\cdot ; U\setminus \ol V))$ given $h|_V$, so $(h,D)$ satisfies condition~\ref{item-equiv-closure}. 
\end{proof}
 
The following lemma is an immediate consequence of Definition~\ref{def-local-metric} and will be important for the proof of Theorem~\ref{thm-msrble-general}.

\begin{lem} \label{lem-open-ind}
Let $(h,D)$ be a coupling of a GFF on $U$ and a random continuous length metric on $U$ which is local for $h$. 
Let $\mcl W$ be a countable collection of disjoint open subsets of $U$. 
Then the internal metrics $\{D(\cdot,\cdot ; W) : W\in \mcl W\}$ are conditionally independent given $h$. 
\end{lem}
\begin{proof} 
By further conditioning on $h|_{U \setminus  V}$ in Definition~\ref{def-local-metric}, we get that if $V\subset \BB C$ is an open set, then $D (\cdot,\cdot ; V)$ and $D(\cdot,\cdot ; U \setminus \ol V)$ are conditionally independent given $h$.  
We now apply this in the case when $V$ is a countable union of sets in $\mcl W$. 
Since the elements of $\mcl W$ are disjoint, we find that $D(\cdot,\cdot ; V)$ coincides with $D(\cdot,\cdot;W)$ on each $W\in\mcl W$ with $W\subset V$ and any two distinct sets $W,W'\in\mcl W$ with $W , W'\subset V$ have infinite $D (\cdot,\cdot;V)$-distance from each other. 
Therefore, $D (\cdot,\cdot;V)$ generates the same $\sigma$-algebra as $\{D(\cdot,\cdot;W) : W\in\mcl W, W\subset V\}$. 
We also note that if $W\not\subset V$, equivalently $W\cap \ol V=\emptyset$, then $D (\cdot,\cdot; W)$ is the internal metric of $D(\cdot,\cdot; U\setminus \ol V)$ on $W$. 
Applying these observations with $V$ ranging over all finite unions of sets in $\mcl W$ gives the lemma statement. 
\end{proof}

\section{Iterating events for local metrics in an annulus}
\label{sec-annulus-iterate}

Throughout this subsection, we let $h$ be a whole-plane GFF normalized so that $h_1(0) = 0$, we fix $\xi \in \BB R$, and we let $D_1,\dots,D_N$ be random metrics on $\BB C$ which are coupled with $h$ and are jointly local and $\xi$-additive for $h$ (Definitions~\ref{def-jointly-local} and~\ref{def-additive-local}).
We will prove the following local independence property for events which depend on $h$ and the metrics $D_1,\dots,D_N$ in concentric annuli, which is a key tool in the proof of Theorem~\ref{thm-bilip} and will also be used in~\cite{lqg-metric-estimates,gm-confluence,gm-uniqueness,gm-coord-change}.  This property is essentially proven in~\cite[Section 4]{mq-geodesics}, but the statements there are given at a slightly lower level of generality so we explain the necessary changes here.  For the statement, we recall the notation for Euclidean annuli from~\eqref{eqn-annulus-def}.

\begin{lem} \label{lem-annulus-iterate}
Fix $0 < s_1<s_2 < 1$. Let $\{r_k\}_{k\in\BB N}$ be a decreasing sequence of positive numbers such that $r_{k+1} / r_k \leq s_1$ for each $k\in\BB N$ and let $\{E_{r_k} \}_{k\in\BB N}$ be events such that each $E_{r_k}$ is measurable w.r.t.\ the $\sigma$-algebra generated by
\eqb \label{eqn-annulus-iterate-triple}
 \left( (h-h_{r_k}(0)) |_{\BB A_{s_1 r_k , s_2 r_k}(0) } , \left\{ e^{-\xi h_{r_k}(0)} D_n \left(\cdot,\cdot ; \BB A_{s_1 r_k, s_2 r_k}(0) \right) \right\}_{n = 1,\dots,N}  \right)  . 
\eqe 
For $K\in\BB N$, let $\mcl N(K)$ be the number of $k\in [1,K]_{\BB Z}$ for which $E_{r_k}$ occurs. 
\begin{enumerate} 
\item For each $a > 0$ and each $b\in (0,1)$, there exists $p = p(a,b,s_1,s_2) \in (0,1)$ and $c = c(a,b,s_1,s_2) > 0$ such that if \label{item-annulus-iterate-high}
\eqb \label{eqn-annulus-iterate-prob}
\BB P\left[ E_{r_k}  \right] \geq p , \quad \forall k \in \BB N  ,
\eqe 
then 
\eqb \label{eqn-annulus-iterate}
\BB P\left[ \mcl N(K)  < b K\right] \leq c e^{-a K} ,\quad\forall K \in \BB N. 
\eqe
\item For each $p\in (0,1)$, there exists $a > 0$, $b\in (0,1)$, and $c > 0$, depending only on $p,s_1,s_2$ such that if~\eqref{eqn-annulus-iterate-prob} holds, then~\eqref{eqn-annulus-iterate} holds. \label{item-annulus-iterate-pos}
\end{enumerate}
\end{lem}

In practice, one most often uses Lemma~\ref{lem-annulus-iterate} to say that it is very likely that at least one of the events $E_{r_k}$ occurs, i.e., we do not care about the particular value of $b$. However, it is occasionally useful to make many of the $E_{r_k}$'s occur, rather than just one.

For $r >0$, we define the $\sigma$-algebra
\eqb \label{eqn-outside-filtration}
\mcl F_r := \sigma\left( (h-h_r(0))|_{\BB C\setminus B_r(0)} , \left\{ e^{-\xi h_r(0)} D_n\left(\cdot,\cdot; \BB C\setminus \ol{B_r(0)} \right)\right\}_{n =1,\dots,N}  \right) ,
\eqe 
which contains all of the information about what happens outside of $B_r(0)$. 
The idea of the proof of Lemma~\ref{lem-annulus-iterate} is to bound the Radon-Nikodym derivative between the conditional law of the $(n+1)$-tuple~\eqref{eqn-annulus-iterate-triple} given $\mcl F_r$ and the marginal law of this $(n+1)$-tuple, and thereby get approximate independence between the events $E_{r_k}$ for different values of $k$. 
For this purposes we need the following elementary observation.

\begin{lem} \label{lem-outside-filtration}
For $r' > r$, we have $\mcl F_{r'} \subset \mcl F_r$. 
\end{lem}
\begin{proof}
The random variable $h_{r'}(0) -h_r(0)$ is equal to the circle average of $h-h_r(0)$ over $\bdy B_{r'}(0)$.
Therefore, $h_{r'}(0) -h_r(0) \in \mcl F_r$. Since $h - h_{r'}(0) = h - h_r(0) - ( h_{r'}(0) -h_r(0))$, we also have $(h-h_{r'}(0))|_{\BB C\setminus B_{r'}(0)} \in \mcl F_r$.
Since $D_n(\cdot,\cdot ; \BB C\setminus \ol{B_{r'}(0)})$ is equal to the internal metric of $D_n(\cdot,\cdot ; \BB C\setminus \ol{B_r(0)})$ on $\BB C\setminus \ol{B_{r'}(0)}$, it follows that also $e^{-\xi h_{r'}(0)} D_n(\cdot,\cdot ; \BB C\setminus \ol{B_{r'}(0)}) \in \mcl F_r$. 
By~\eqref{eqn-outside-filtration}, we now get $\mcl F_{r'} \subset\mcl F_r$. 
\end{proof}
 
By Lemma~\ref{lem-whole-plane-markov}, for each $r > 0$ we can write $(h-h_r(0))|_{B_r(0)} = \frk h^r + \rng h^r$, where $\frk h^r$ is a random harmonic function on $B_r(0)$ which is determined by $(h-h_r(0))|_{\BB C\setminus B_r(0)}$ and $\rng h^r$ is a zero-boundary GFF on $B_r(0)$ which is independent from $(h-h_r(0))|_{\BB C\setminus B_r(0)}$. 
  
Our Radon-Nikodym derivative will be in terms of the fluctuation of the harmonic part of $\frk h^r$ on a smaller ball: for $0 < r < R$, let
\eqb \label{eqn-max-def}
\frk M_r^R := \sup_{z  \in B_{r}(0)} |\frk h^R(z) - \frk h^R(0)| .
\eqe
Note that $\frk M_r^R$ is determined by $\frk h^R$ and hence by $(h-h_R(0))|_{\BB C\setminus B_R(0)}$.

\begin{lem} \label{lem-gff-rn}
Fix $0 < s < s' < 1$ and let $\mcl F_r$ be as in~\eqref{eqn-outside-filtration}. For $r> 0$, a.s.\ the conditional law given $\mcl F_r$ of the $(n+1)$-tuple
\eqb \label{eqn-annulus-iterate-triple'}
 \left( (h-h_{ r}(0)) |_{B_{s r}(0) } , \left\{ e^{-\xi h_{r}(0)} D_n \left(\cdot,\cdot ; B_{s r}(0) \right) \right\}_{n = 1,\dots,N}  \right)  
\eqe  
is absolutely continuous with respect to its marginal law. Furthermore, for each $\alpha > 0$ and $M  >0$ there exists $C = C(\alpha, M , s,s') > 0$ such that such that if $H_r   $ denotes the Radon-Nikodym derivative of the conditional law with respect to the marginal law, then on the ($\mcl F_r$-measurable) event $\{\frk M_{s' r}^r \leq M\}$, a.s. 
\eqb \label{eqn-rn-moment}
\max\left\{ \BB E\left[ H_r^{ \alpha} \,|\,  \mcl F_r   \right] , \, \BB E\left[ H_r^{-\alpha} \,|\, \mcl F_r \right] \right\}      
\leq C .
\eqe 
\end{lem}
\begin{proof} 
By $\xi$-additivity, the metrics $e^{-\xi h_r(0)} D_n   $ for $n=1,\dots,N$ are jointly local for $h-h_r(0)$. 
Therefore, the metrics $ \left\{  e^{-\xi h_r(0)} D_n\left(\cdot,\cdot;  B_{s r}(0)  \right)\right\}_{n=1,\dots,N}$ are conditionally independent from $\mcl F_r$ given $(h-h_r(0) )|_{B_{s r}(0)}$. 
We therefore only need to compare the conditional law of $(h-h_r(0)) |_{B_{s r }(0)}$ given $\mcl F_r$ to the marginal law of $(h-h_r(0))|_{B_{s r}(0)}$.  
Again by locality, the conditional law of $(h-h_r(0)) |_{B_{s r }(0)}$ given $\mcl F_r$ depends only on $ (h-h_r(0)) |_{\BB C\setminus B_r(0)}$. 
We have therefore reduced to estimating the Radon-Nikodym derivative of the conditional law of $(h-h_r(0)) |_{B_{s r }(0)}$ given $ (h-h_r(0)) |_{\BB C\setminus B_r(0)}$ with respect to the marginal law of $(h-h_r(0))|_{B_{sr}(0)}$.
By the scale invariance of the law of the GFF, modulo additive constant, it suffices to estimate this law in the case when $r = 1$.   
This is a standard calculation for the GFF and is carried out in~\cite[Lemma 4.1]{mq-geodesics} in the special case when $s = 7/8$ and $s' = 15/16$. 
The same proof works for a general choice of $s$ and $s'$. 
\end{proof}

The following lemma will allow us to apply Lemma~\ref{lem-gff-rn} at a dense set of scales.

\begin{lem} \label{lem-good-radius-harmonic}
Fix $s\in (0,1)$ and let $\{r_k\}_{k\in\BB N}$ be a decreasing sequence of positive numbers such that $r_{k+1} / r_k \leq s$ for each $k\in\BB N$.
For $K \in \BB N$ and $M > 0$, let $\mcl N_M(K)$ be the number of $k\in [1,K]_{\BB Z}$ for which $\frk M_{s r}^r \leq M$. 
For each $a > 0$ and each $b\in (0,1)$, there exists $M , c  > 0$, depending only on $a$, $b$, and $s$, such that 
\eqb
\BB P\left[ \mcl N_M(K) \geq b K \right] \geq 1 - c e^{-a K},\quad\forall K \in\BB N .
\eqe
\end{lem}
\begin{proof}
This follows from exactly the same argument used to prove~\cite[Proposition 4.3]{mq-geodesics}, although~\cite[Proposition 4.3]{mq-geodesics} is only stated in the special case when $r_k = 2^{-k} r$ for some $r\in (0,1)$.
\end{proof}

For the proof of Lemma~\ref{lem-annulus-iterate}, we will also need the following elementary tail estimate for the binomial distribution; see, e.g.,~\cite[Lemma~2.6]{mq-geodesics}.

\begin{lem} \label{lem-binom-tail}
Let $p \in (0,1)$ and $n\in\BB N$ and let $B_n$ be a random variable with the binomial distribution with parameters $n$ and $p$. 
For $\alpha \in (0,p)$, 
\eqb \label{eqn-binom-tail}
\BB P\left[ B_n <  \alpha n \right] \leq   \left( \frac{1 - p}{1-\alpha} \right)^n \left( \frac{(1-\alpha ) p  }{ \alpha(1-p) }  \right)^{\alpha n} = e^{-c_{p,\alpha} n} , 
\eqe
where $c_{p,\alpha} > 0$ satisfies $c_{p,\alpha} \rta \infty$ as $p\rta 1$ ($\alpha$ fixed). 
\end{lem}

\begin{proof}[Proof of Lemma~\ref{lem-annulus-iterate}] 
Set $s' := (1+s_2)/2 \in (s_2,1)$. Also fix a parameter $M \geq 1$ which we will eventually choose to be sufficiently large, depending on $a,b,s_1,s_2$ in the case of assertion~\ref{item-annulus-iterate-high} or $p,s_1,s_2$ in the case of assertion~\ref{item-annulus-iterate-pos}. 
By Lemma~\ref{lem-gff-rn} and H\"older's inequality and since $E_{r_k}$ is determined by the $(n+1)$-tuple~\eqref{eqn-annulus-iterate-triple}, we can find $p_M = p_M(p, s_1,s_2,M) \in (0,1)$ such that the following is true.
If~\eqref{eqn-annulus-iterate-prob} holds, then for each $k \in \BB N$,  
\eqb \label{eqn-good-event-prob}
\BB P\left[ E_{r_k} \,|\, \mcl F_{r_k} \right] \geq p_M ,\quad \text{a.s.\ on the event} \quad \{\frk M^{r_k}_{s' r_k} \leq M\}. 
\eqe
Furthermore, for any $\delta >0$ there exists $p_*  = p_*(\delta , s_1,s_2,M) \in (0,1)$ such that if $p\geq p_*$, then $p_M \geq 1-\delta$.  
As in the proof of Lemma~\ref{lem-outside-filtration}, we have $h_{r_k}(0) - h_{ s_1 r_k}(0) \in \mcl F_{s_1 r_k}$ and hence the triple~\eqref{eqn-annulus-iterate-triple} is $\mcl F_{s_1 r_k}$-measurable. By this and our measurability hypothesis for the $E_{r_k}$'s, and the fact that $r_{k+1} / r_k \leq s_1$, we infer that
\eqb \label{eqn-good-event-msrble}
E_{r_k} \in \mcl F_{s_1 r_k} \subset \mcl F_{r_{k+1}} . 
\eqe

For $j \in \BB N$, let $k_j$ be the $j$th smallest $k\in\BB N$ for which $\frk M_{s' r_k}^{r_k} \leq M$. 
By Lemma~\ref{lem-good-radius-harmonic} applied with $s = s'$ and $b^{-1/2} a$ in place of $a$, for a given choice of $a  > 0$ and $b\in(0,1)$ we can find $M ,  c_0 > 0$ depending only on $a,b,s_1$ such that (in the notation of that lemma) 
\eqb \label{eqn-use-harmonic} 
\BB P\left[ k_j  >  b^{ - 1/2} j \right] = \BB P\left[ \mcl N_M(\lfloor b^{-1/2} j \rfloor)  < j \right] \leq c_0 e^{- a b^{-1/2}  j}  ,\quad\forall j \in \BB N . 
\eqe
In the setting of assertion~\ref{item-annulus-iterate-high}, we henceforth fix $M$ so that~\eqref{eqn-use-harmonic} holds for the given choice of $a,b$. 
In the setting of assertion~\ref{item-annulus-iterate-pos}, we instead choose $M$ so that~\eqref{eqn-use-harmonic} holds for $a = 1$ and $b =1/2$, say. 

Each $k_j$ is a stopping time for $\{\mcl F_{r_k}\}_{k\in\BB N}$. 
By~\eqref{eqn-good-event-prob}, for $j\in\BB N$, a.s.\ 
\eqbn
\BB P\left[ E_{r_{k_j} } \,|\, \mcl F_{r_{k_j}} \right] \geq p_M  .
\eqen
Combining this with~\eqref{eqn-good-event-msrble} shows that for $j \in \BB N$, the number $\mcl N(k_j)$ as defined in the lemma statement with $K = k_j$ stochastically dominates a binomial distribution with $j$ trials and success probability $p_M$.

Since $p_M$ can be made arbitrarily close to 1 by choosing $p$ to be sufficiently close to 1 (depending on $M$), it follows from Lemma~\ref{lem-binom-tail} (applied with $\alpha = b^{1/2}$) that for each $a > 0$ and $b\in(0,1)$, there exists $p \in (0,1)$ and $c_1 > 0$, depending only on $a,b,s_1,s_2,M$, such that if~\eqref{eqn-annulus-iterate-prob} holds, then 
\eqb \label{eqn-binom-compare}
\BB P\left[ \mcl N(k_j) < b^{1/2} j \right] \leq c_1 e^{-a b^{-1/2} j} ,\quad\forall j \in \BB N .
\eqe
Furthermore, if~\eqref{eqn-annulus-iterate-prob} holds for some choice of $p \in (0,1)$, then since $p_M > 0$, it follows that~\eqref{eqn-binom-compare} holds for some choice of $a  ,b , c_1$ (depending on $p , s_1,s_2,M$). 

In the setting of assertion~\ref{item-annulus-iterate-high}, for a given $K\in\BB N$, we now set $j = \lfloor b^{ 1/2} K \rfloor$ and combine~\eqref{eqn-use-harmonic} with~\eqref{eqn-binom-compare} to get that 
\eqb
\BB P\left[ \mcl N(K) < b K  \right] 
\leq \BB P\left[ \mcl N(k_j) < b^{1/2} j \right] + \BB P\left[ k_j > K \right] 
\leq (c_0 + c_1) e^{- a   K} . 
\eqe
This gives assertion~\ref{item-annulus-iterate-high}. We similarly obtain assertion~\ref{item-annulus-iterate-pos} by combining~\eqref{eqn-use-harmonic} and~\eqref{eqn-binom-compare}.
\end{proof}

\section{Bi-Lipschitz equivalence}
\label{sec-bilip}

In this section we will prove Theorem~\ref{thm-bilip}. 
Throughout, we assume that we are in the setting of that theorem, so that $h$ is a whole-plane GFF normalized so that $h_1(0) = 0$ and $(D,\wt D)$ are jointly local, $\xi$-additive metrics for $h|_U$. 
Let $C > 0$ be as in~\eqref{eqn-bilip}. 
For $z\in U$, and $r > 0$ such that $B_r(z)\subset U$, let 
\eqb \label{eqn-good-dist-event}
E_r(z ) := \left\{\sup_{u,v \in \bdy B_r(z)} \wt D\left(u,v; \BB A_{r/2,2r}(z) \right) \leq C D(\bdy B_{r/2}(z) , \bdy B_r(z) )  \right\} , 
\eqe
so that if $K \subset U$ is compact, then $\BB P[E_r(z)] \geq p$ for all $z\in K$ and all $r\in (0,r_K]$. 
We think of annuli $\BB A_{r/2,  2 r}(z) $ for which $E_r(z )$ occurs as ``good". We will eventually show that with high probability every point in $K$ is contained in a ball of the form $B_{r/2}(z)$ for which $E_r(z)$ occurs. 
Stringing together paths in such balls leads to a proof of Theorem~\ref{thm-bilip}. 
See Figure~\ref{fig-bilip} for an illustration and outline of the proof.  
The main estimate which we need is the following lemma, which is a consequence of Lemma~\ref{lem-annulus-iterate}.

\begin{figure}[ht!]
\begin{center}
\includegraphics[scale=.8]{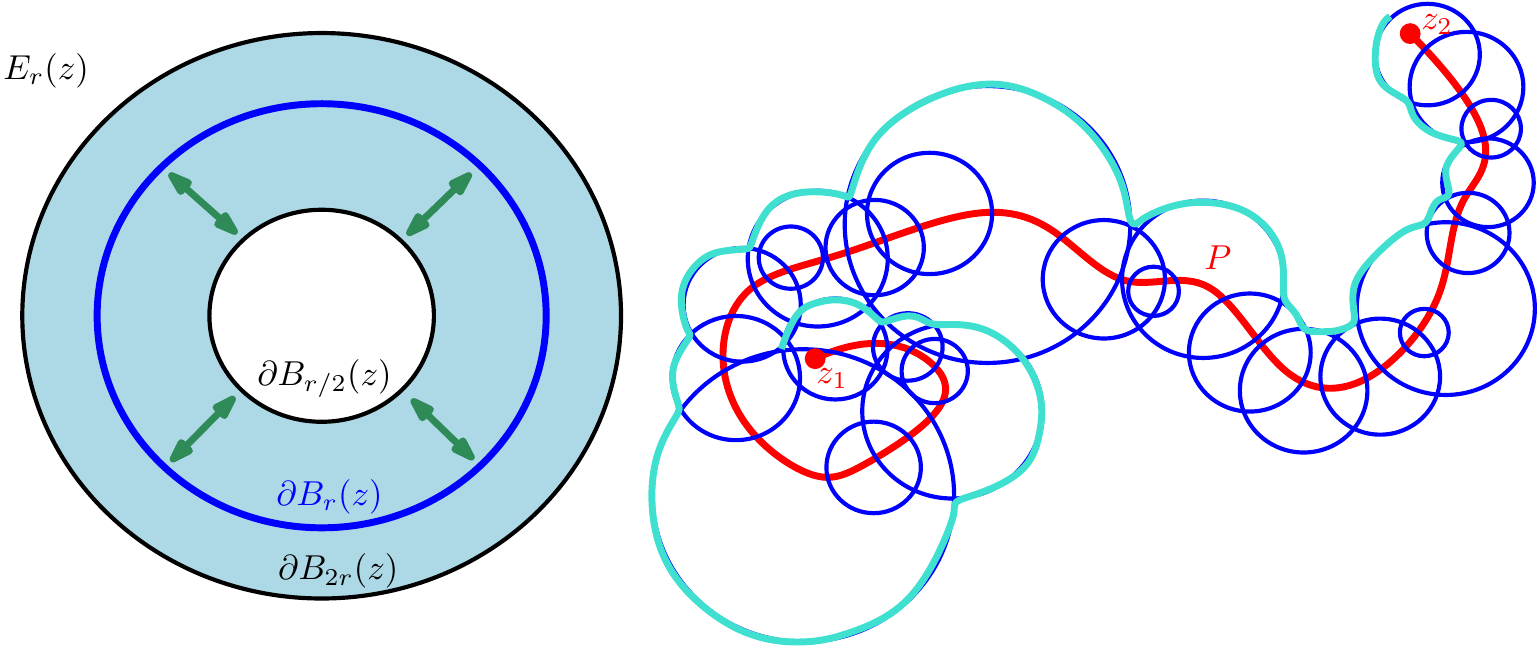} 
\caption{\label{fig-bilip} 
\textbf{Left:} A ``good" annulus $\BB A_{r/2,2r}(z)$, i.e., one for which $E_r(z)$ occurs. The diameter of the blue circle $\bdy B_r(z)$ w.r.t.\ the $\wt D$-internal metric on $\BB A_{r/2,2r}(z)$ is at most $C$ times the $D$-distance across the inner annulus $\BB A_{r/2,r}(z)$.
By Lemma~\ref{lem-good-radius-all}, if $C$ is taken to be large enough then for any fixed compact set $K\subset\BB C$ it holds with high probability when  $\ep$ is small that every point in $K$ is contained in a ball $B_{r/2}(w)$ where $r   \in [ \ep^2 , \ep] \cap \{2^{-k}\ep  :k\in\BB N\}$ and $w\in (\frac14 \ep^2 \BB Z^2) \cap B_\ep(K)$ are such that $E_r(z)$ occurs. 
\textbf{Right:} We consider a $D$-geodesic $P$ (shown in red) between arbitrary points $z_1,z_2$ and look at the successive times $t_j$, $j\in [1,J]_{\BB Z}$ at which $P$ crosses an annulus $\BB A_{r/2,r}(w)$ where $z$ and $r$ are as above. By Lemma~\ref{lem-connected} the union of the outer circles $\bdy B_r(w)$ for these annuli (blue) contains a path (teal) from $B_{2\ep}(z_1)$ to $B_{2\ep}(z_2)$. By the triangle inequality and the definition of a good annulus, this shows that $\wt D(z_1,z_2)$ is at most $C D(z_1,z_2)$ plus a small error which tends to zero as $\ep\rta 0$. 
}
\end{center}
\end{figure}

\begin{lem} \label{lem-good-radius-tail}
For $z\in U$, and $r > 0$ such that $B_r(z)\subset U$, let $\rho_r(z)$ be the largest $t \in [0,r]$ such that $t = 2^{-k} r$ for some $k\in\BB N$ and $E_t(z)$ occurs.
For each $q > 0$, there is a constant $p_q > 0$ depending only on $q$ such that the following is true.
If $K\subset U$ is compact and~\eqref{eqn-bilip} holds for some $p \geq p_q$ and $r_K>0$, then for each $r\in (0,r_K]$,
\eqb \label{eqn-good-radius-tail}
\BB P\left[ \rho_r(z)  < \ep r \right]  = O_\ep(\ep^q) ,\quad \text{as $\ep\rta 0$} ,
\eqe  
at a rate depending only on $q$. 
\end{lem}
\begin{proof}
Since scaling each of $D$ and $\wt D$ by the same constant factor does not affect the occurrence of $E_r(z)$, it follows that $E_r(z)$ is a.s.\ determined by $e^{-\xi h_r(z)} D \left(\cdot,\cdot ; \BB A_{  r/2 ,  2 r}(z) \right) $ and $ e^{-\xi h_r(z)} \wt D \left(\cdot,\cdot ; \BB A_{ r/2, 2 r}(z) \right) $. 
Furthermore, by the $\xi$-additivity of $(D,\wt D)$ and the fact that the locality condition is preserved under translating and scaling space, it follows that $e^{-\xi h_r(z)} D( r\cdot+z , r\cdot+z)$ and $e^{-\xi h_r(z)} \wt D(r\cdot+z , r\cdot+z) $ are jointly local metrics for the field $(h(r\cdot+z) - h_r(z))|_{r^{-1}(U-z)}$.
The field $h(r\cdot+z) - h_r(z)$ has the same law as $h$, so we can apply Lemma~\ref{lem-annulus-iterate} to this field (with $r_k = 2^{-k} r$, any choice of $b\in(0,1)$, $K = \lfloor \log_2\ep^{-1} \rfloor$, and $a = q\log 2$) to get the statement of the lemma.  
\end{proof}

Applying Lemma~\ref{lem-good-radius-tail} a large finite number of times leads to the following. 

\begin{lem} \label{lem-good-radius-all}
There exists a universal constant $p \in (0,1)$ such that if $K\subset U$ is compact and there exists $r_K > 0$ such that~\eqref{eqn-bilip} holds with this choice of $p$, then it holds with probability tending to 1 as $\ep\rta 0$ (at a rate which depends on $r_K$ and $K$) that the following is true.
For each $z\in K$, there exists $r   \in [ \ep^2 , \ep] \cap \{2^{-k}\ep  :k\in\BB N\}$ and $w\in (\frac14 \ep^2 \BB Z^2) \cap B_\ep(K)$ such that $z\in B_{r/2}(w)$ and $E_r(w )$ occurs. 
\end{lem}
\begin{proof} 
By Lemma~\ref{lem-good-radius-tail} applied with $q = 5$, say, and a union bound over all $w\in (\frac14 \ep^2 \BB Z^2) \cap B_\ep(K)$, if $p\geq p_5$ then it holds with probability tending to 1 as $\ep\rta 0$ that $\rho_\ep(w) \in [\ep^2 , \ep]$ for every such $w$. Since the balls $B_{\ep^2/2}(w)$ for $w\in (\frac14 \ep^2 \BB Z^2) \cap B_\ep(K)$ cover $K$, the statement of the lemma follows. 
\end{proof}

We now turn our attention to the proof of Theorem~\ref{thm-bilip}. 
Fix $z_1 , z_2 \in \BB C$. We will show that a.s.
\eqb \label{eqn-metric-compare-show}
\wt D(z_1, z_2) \leq C D(z_1,z_2) .
\eqe
This implies that a.s.\ \eqref{eqn-metric-compare-show} holds simultaneously for every $z_1,z_2\in\BB Q^2$. By the continuity of $(z_1,z_2) \mapsto  D(z_1,z_2)$ and $(z_1,z_2) \mapsto  \wt D(z_1,z_2)$, it follows that a.s.~\eqref{eqn-metric-compare-show} holds for every $z_1,z_2\in\BB C$ simultaneously.   
Thus we only need to prove~\eqref{eqn-metric-compare-show} for an arbitrary fixed choice of $z_1,z_2\in\BB C$. 

To this end, fix a small $\delta > 0$ (which we will eventually send to zero) and let $P : [0, T] \rta U$ be a path from $z_1$ to $z_2$ with $D$-length smaller than $D(z_1,z_2) + \delta$, chosen in a measurable manner. We assume that $P$ is parameterized by its $D$-length. 
Since the range of $P$ is a compact subset of $U$, we can find compact set $K\subset U$ such that $\BB P\left[ P \subset K\right]\geq 1 - \delta$.  
For $\ep  >0$, let $F^\ep$ be the event of Lemma~\ref{lem-good-radius-all} with this choice of $K$, so that $\BB P[F^\ep] \rta 1$ as $\ep\rta 0$. 
We will work on the event $\{P\subset K\} \cap F^\ep$, which happens with probability tending to 1 as $\ep\rta 0$ and then $\delta\rta 0$. 

Let $t_0  = 0$ and inductively let $t_j$ for $j\in\BB N$ be the smallest time $t \geq t_{j-1}$ at which $P$ exits a Euclidean ball of the form $B_r(w)$ for $w\in (\frac14 \ep^2 \BB Z^2 ) \cap B_\ep(K)$ and $r\in [\ep^2 , \ep] \cap \{2^{-k}\ep : k\in\BB N\}$ such that $z\in B_{r/2}(w)$ and $E_r(w )$ occurs; or let $t_j = D(z_1 , z_2 )$ if no such $t$ exists.  
If $t_j  < D(z_1 , z_2 )$, let $w_j$ and $r_j$ be the corresponding values of $r$ and $w$. 
Let $J$ be the smallest $j\in\BB N_0$ for which $t_{j+1} = D(z_1 , z_2  )$.

The definition of $F^\ep$ implies that a path in $K$ cannot travel Euclidean distance further than $2\ep$ without crossing one of the annuli $\BB A_{r/2,   r}(w) $ with $w\in (\frac14 \ep^2 \BB Z^2 ) \cap B_{\ep}(K) $ and $r\in [\ep^2 , \ep] \cap \{2^{-k}\ep : k\in\BB N\}$ such that $E_r(w )$ occurs.
Since $P$ is a path from $z_1$ to $z_2$, it follows that 
\eqb \label{eqn-endpoint-close}
|P(t_J) - z_2| \leq 2\ep . 
\eqe
By the definition of $E_{r_j}(w_j )$ and since $P$ is parameterized by $D$-length and crosses $\BB A_{r_j/2, r_j}(w_j)$ during the time interval $[t_{j-1} , t_j]$, for each $j\in [1,J]_{\BB Z}$ one has
\eqb \label{eqn-good-radius-compare}
\sup_{u,v \in \bdy B_{r_j}(w_j)} \wt D \left( u,v ;   \BB A_{r_j/2,  2r_j}(w_j)   \right)  
\leq C D \left(\bdy B_{ r_j/2}(w_j) , \bdy B_{ r_j}(w_j)   \right) 
\leq  C (t_j - t_{j-1}) .
\eqe  
In order to use this to get an upper bound for $\wt D (z_1,z_2)$ in terms of $D (z_1,z_2)$, we need the following elementary topological lemma.

\begin{lem} \label{lem-connected}
On the event $\{P\subset K\} \cap F^\ep$, the union of the circles $\bdy B_{r_j}(w_j)$ for $j\in [1,J]_{\BB Z}$ contains a path from $B_{2\ep}(z_1)$ to $B_{2\ep}(z_2)$. 
\end{lem}
\begin{proof}   
By definition, the union of the balls $B_{r_j}(w_j)$ for $j\in [1,J]_{\BB Z}$ covers $P([0,t_J))$, and each such ball has radius at most $\ep$. 
Let $\mcl B$ be a sub-collection of the balls $B_{r_j}(w_j)$ for $j\in [1,J]_{\BB Z}$ which is minimal in the sense that $ P([0,t_J))  \subset\bigcup_{B\in\mcl B} B$ and $ P([0,t_J))$ is not covered by any proper subset of the balls in $\mcl B$. 
Since $ P([0,t_J))$ is connected, it follows that $\bigcup_{B\in\mcl B} B$ is connected. Indeed, if this set had two proper disjoint open subsets, then each would have to intersect $ P([0,t_J))$ (by minimality) which would contradict the connectedness of $ P([0,t_J))$. 
Furthermore, by minimality, no ball in $\mcl B$ is properly contained in another ball in $\mcl B$. 

We claim that $\bigcup_{B\in\mcl B} \bdy B$ is connected. Indeed, if this were not the case then we could partition $\mcl B = \mcl B_1\sqcup \mcl B_2$ such that $\mcl B_1$ and $\mcl B_2$ are non-empty and $\bigcup_{B\in\mcl B_1} \bdy B$ and $\bigcup_{B\in\mcl B_2} \bdy B$ are disjoint.  
By the minimality of $\mcl B$, it cannot be the case that any ball in $\mcl B_2$ is contained in $ \bigcup_{B\in\mcl B_1}  B$.
Furthermore, since $\bigcup_{B\in\mcl B_1} \bdy B$ and $\bigcup_{B\in\mcl B_2} \bdy B$ are disjoint, it cannot be the case that any ball in $\mcl B_2$ intersects both $\bigcup_{B\in\mcl B_1} B$ and  $\BB C\setminus \bigcup_{B\in\mcl B_1}  B$ (otherwise, such a ball would have to intersect the boundary of some ball in $\mcl B_1$). 
Therefore, $\bigcup_{B\in\mcl B_1}  B$ and $\bigcup_{B\in\mcl B_2} \bdy B$ are disjoint. 
Since no element of $\mcl B_1$ can be contained in $\bigcup_{B\in\mcl B_2}  B$, we get that $\bigcup_{B\in\mcl B_1}  B$ and $\bigcup_{B\in\mcl B_2}  B$ are disjoint.
This contradicts the connectedness of $\bigcup_{B\in\mcl B} B$, and therefore gives our claim. 

Since $\bigcup_{B\in\mcl B} B$ contains $ P([0,t_J))$, $P(t_J) \in B_{2\ep}(z_2)$, and each ball in $\mcl B$ has radius at most $\ep$, it follows that $\bigcup_{B\in\mcl B} \bdy B$ contains a path from $B_{2\ep}(z_1)$ to $B_{2\ep}(z_2)$, as required. 
\end{proof}

By~\eqref{eqn-good-radius-compare}, Lemma~\ref{lem-connected}, and the triangle inequality, on the event $\{P\subset K\} \cap F^\ep$,  
\eqbn
\wt D \left(B_{2\ep}(z_1) , B_{2\ep}(z_2)  \right) \leq C \sum_{j=1}^J (t_j - t_{j-1}) \leq C \left( D (z_1 , z_2  )  + \delta \right) .
\eqen
Since $\wt D$ is a continuous function on $\BB C\times\BB C$, a.s.\ $\wt D\left(B_{2\ep}(z_1) , B_{2\ep}(z_2)  \right)  \rta \wt D (z_1,z_2)$ as $\ep\rta 0$. 
Since $\BB P[ \{P\subset K\} \cap F^\ep ] \rta 1$ as $\ep\rta 0$ and then $\delta \rta 0$, a.s.~\eqref{eqn-metric-compare-show} holds. \qed

\section{Measurability}
\label{sec-msrble}

In this section we will prove Theorem~\ref{thm-msrble-general}.  
Throughout, we assume that we are in the setting of Theorem~\ref{thm-msrble-general}.

The key tool in the proof is the Efron-Stein inequality~\cite{efron-stein}, which says that if $F = F(X_1,\dots,X_n)$ is a measurable function of $n$ independent random variables, then
\eqb \label{eqn-efron-stein}
\op{Var}[F] \leq \sum_{i=1}^n \op{Var}\left[ F \,|\, \{X_j\}_{j\not=i} \right] .
\eqe
To apply~\eqref{eqn-efron-stein} in our setting, we will divide $  U$ into a fine square grid (which will be randomly shifted, for technical reasons; see Lemma~\ref{lem-random-grid}) and use the locality of $D$ to get that the internal metrics of $D $ on the squares of this grid are conditionally independent given $h$. We will also show that $D$ is a.s.\ determined by this internal metrics (Lemma~\ref{lem-internal-msrble}). We then fix $z,w\in U$ and apply~\eqref{eqn-efron-stein} to the conditional law of the random variable $F = D (z,w)$ given~$h$. To do this we need to bound the conditional variance when we re-sample the internal metric on one square $S$. For this purpose, we will consider a path $P$ from $z$ to $w$ in $U$ of near-minimal $D$-length and use our bi-Lipschitz hypothesis to get that the difference between the original value of $D(z,w)$ and the new value when we re-sample in $S$ is at most a constant times the $D $-length of $P \cap S$. 
When we send the mesh size to zero, the sum over all $S$ of the squared error $\op{len}(P\cap S)^2$ will converge to zero a.s., which will show that $\op{Var}\left[ D(z,w) \,|\, h \right] = 0$ and hence that $D $ is a.s.\ determined by $h$. 
 
We will need a few preparatory lemmas. The following is essentially a re-formulation of our bi-Lipschitz equivalence hypothesis. 
The lemma implies in particular that if we condition on $h$ and sample two metrics from the conditional law of $D$ given $h$, then a.s.\ the two metrics are bi-Lipschitz equivalent, even if we do not assume that the metrics are conditionally independent given $h$.

\begin{lem} \label{lem-cond-bded}
Assume we are in the setting of Theorem~\ref{thm-msrble-general}, i.e., $h$ is a GFF on $U\subset \BB C$, $D$ is a local metric for $h$, and there is a random constant $C = C_h >1$, depending only $h$, such that the following is true. If $D,\wt D$ are conditionally independent samples from the conditional law of $D$ given $h$ then a.s.\ $ \wt D (z,w) \leq C D (z,w)$ for each $z,w\in U$.  
Fix a connected open set $V\subset U$ and distinct points $z,w\in  V$.  
Then a.s.\ $C^{-1} \BB E[D (z,w;V)\,|\,h] \leq D(z,w;V) \leq C \BB E[D (z,w;V)\,|\,h]$. 
\end{lem}
\begin{proof}
Condition on $h$ and sample $D$ and $\wt D$ conditionally independently from the conditional law of $D$ given $h$.
By our hypothesis for $D$ and $\wt D$ from Theorem~\ref{thm-msrble-general}, a.s.\ the $\wt D$-length of any path in $U$ is at most $C$ times its $D$-length. 
In particular, a.s.\ $ \wt D (z,w;V) / D (z,w;V) \leq C$. 
Let 
\eqbn
m_h := \inf\left\{ t \geq 0 : \BB P\left[D(z,w;V) < t \,|\, h \right] > 0 \right\} \quad \text{and} \quad
M_h := \sup\left\{ t \geq 0 : \BB P\left[D(z,w;V) > t \,|\, h \right] > 0 \right\} .
\eqen
Then $m_h$ and $M_h$ are determined by $h$ and a.s.\ each of $D (z,w ; V)$ and $\BB E[D(z,w;V)\,|\,h]$ is in $[m_h , M_h]$. 
To prove the lemma it therefore suffices to show that a.s.\ $M_h/m_h \leq C$. 
 
For any $ A < M_h / m_h$, we can choose $A_1 > m_h$ and $A_2 < M_h$ such that $A_2/A_1 \geq A$. By the conditional independence of $D$ and $\wt D$ given $h$ and the definitions of $m_h$ and $M_h$, 
\eqb
\BB P\left[ \wt D (z,w; V) / D (z,w; V) > A \,|\, h \right] 
\geq \BB P\left[ \wt D (z,w; V) >  A_2  \,|\, h \right]  
\BB P\left[  D (z,w; V) < A_1 \,|\, h \right]  
>0 .
\eqe
Since a.s.\ $\wt D(z,w; V) / D(z,w; V) \leq C$, we must have $A \leq C$ so a.s.\ $M_h / m_h \leq C$. 
\end{proof}

We now define the fine square grid which we will work with. 
Let $\theta $ be sampled uniformly from Lebesgue measure on $[0,1]^2$, independently from everything else. Let $\mcl G_\theta$ be the randomly shifted square grid which is the union of all of the horizontal and vertical line segments joining points of $\BB Z^2+\theta$. 
The reason for the random index shift $\theta$ is to make the following lemma true.

\begin{lem} \label{lem-random-grid} 
Let $P : [0,|P|] \rta  U$ be a random curve with finite $D $-length chosen in a manner depending only on $(h,D )$ (not on $\theta$). 
For each $\ep > 0$, a.s.\ $\op{len}(P ; D ) = \op{len}( P \setminus (\ep\mcl G_\theta) ;  D )$. 
\end{lem}

We note that $P\setminus (\ep\mcl G_\theta)$ is a countable union of excursions of $P$ into $\BB C\setminus (\ep\mcl G_\theta)$, so its $D$-length is well-defined.

\begin{proof}[Proof of Lemma~\ref{lem-random-grid}]
Assume without loss of generality that $P$ is parameterized by $D$-length. For each fixed $t\in [0,\op{len}(P;D )]$ (chosen in a manner depending only on $P$), we have $\BB P[P(t) \in \ep\mcl G_\theta \,|\, P] = 0$ since $P$ is independent from $\theta$. Hence a.s.\ the Lebesgue measure of $P^{-1}(\ep\mcl G_\theta)$ is zero.
\end{proof}

For $\ep > 0$, let $\mcl S_\theta^\ep $ be the set of open $\ep\times\ep$ squares which are the connected components of $\BB C\setminus (\ep\mcl G_\theta)$ and which intersect $U$.
As a consequence of Lemma~\ref{lem-random-grid}, if $P: [0,|P|] \rta U$ is a path as in that lemma then a.s. 
\eqb \label{eqn-geodesic-grid}
\op{len}(P ; D)  = \sum_{S\in\mcl S_\theta^\ep} \op{len}( P \cap S ;  D ) .
\eqe
In fact, we can a.s.\ recover $D$ from its internal metrics on the squares $S\in\mcl S_\theta^\ep$, as the following lemma demonstrates.

\begin{lem} \label{lem-internal-msrble}
The metric $D$ is a.s.\ determined by $h,\theta$, and the set of internal metrics $\{ D(\cdot,\cdot ; S\cap U) : S\in \mcl S_\theta^\ep\}$. 
\end{lem}
\begin{proof}
Condition on $h,\theta$ and $\{ D(\cdot,\cdot ; S\cap U) : S\in \mcl S_\theta^\ep\}$ and let $D$ and $D'$ be two conditionally independent samples from the conditional law of $D$, so that a.s.\ $ D(\cdot,\cdot ; S\cap U) =  D'(\cdot,\cdot ; S\cap U)$ for each $S\in\mcl S_\theta^\ep$. 
To prove the lemma it suffices to show that a.s.\ $D = D'$. 

We first observe that since $(h,D)\eqD (h,D')$, Lemma~\ref{lem-cond-bded} (applied with $V = U$) implies that for each fixed $z,w\in U$, a.s.\ 
\eqbn
D(z,w) , D'(z,w) \in \left[ C^{-1 } \BB E[D(z,w )\,|\,h] ,  C \BB E[D(z,w )\,|\,h] \right] .
\eqen
This holds a.s.\ for all $z,w\in U\cap \BB Q^2$ simultaneously, so since $D$ and $D'$ are continuous metrics, a.s.\ 
\eqb \label{eqn-internal-msrble-bilip}
C^{-2} D (z,w) \leq D'(z,w) \leq C^2 D(z,w) ,\quad \forall z,w\in U .
\eqe 
Now fix $z,w\in U$ and $\delta >0$ and let $P$ be a path from $z$ to $w$ with $D$-length at most $D(z,w)+\delta$, chosen in a manner depending only on $D$. 
By Lemma~\ref{lem-random-grid}, a.s.\ 
\eqb \label{eqn-internal-metric-bdy}
\lim_{\zeta \rta 0} \op{len}\left(P\cap B_\zeta(\ep \mcl G_\theta ) ; D \right)  = 0 ,\quad \text{($\ep,\delta$ fixed)} . 
\eqe 
By~\eqref{eqn-internal-msrble-bilip}, we infer that~\eqref{eqn-internal-metric-bdy} also holds with $D'$-length instead of $D$-length. 
Consequently, a.s.\ 
\eqb \label{eqn-internal-metric-sum}
\op{len}(P ; D) = \sum_{S\in \mcl S_\theta^\ep} \op{len}\left( P \cap S ; D \right) 
\quad \text{and} \quad 
\op{len}(P ; D') = \sum_{S\in \mcl S_\theta^\ep} \op{len}\left( P \cap S ; D' \right) .
\eqe

Since the internal metrics of $D$ and $D'$ on $S\cap U$ coincide for each $S\in\mcl S_\theta^\ep$, a.s.\ the $D'$-length of every path which is contained in some $S\in\mcl S_\theta^\ep$ is the same as its $D$-length. Therefore,~\eqref{eqn-internal-metric-sum} implies that a.s.\ $\op{len}(P; D)  = \op{len}(P; D')$. 
Since $D'$ is a length metric and by our choice of $P$, we have $D'(z,w) \leq D(z,w) +\delta$. Since $\delta >0$ is arbitrary, a.s.\ $D'(z,w) \leq D(z,w)$. Symmetrically, a.s.\ $D(z,w) \leq D'(z,w)$. Applying this for all $z,w\in\BB Q^2\cap U$ now shows that a.s.\ $D = D'$, as required. 
\end{proof}

Lemma~\ref{lem-internal-msrble} together with the following lemma will allow us to express $D$ as a function of a collection of random variables which are conditionally independent given $(h,\theta)$, so that we can apply the Efron-Stein inequality.

\begin{lem} \label{lem-square-ind}
Fix $\ep > 0$. Under the conditional law given $(h, \theta)$, a.s.\ the internal metrics $\{ D(\cdot,\cdot ;S \cap U) : S\in \mcl S_\theta^\ep\}$ are conditionally independent. 
\end{lem}
\begin{proof}
We condition on $\theta$, which determines $\mcl S_\theta^\ep$, then apply Lemma~\ref{lem-open-ind} to the collection of disjoint open sets $\mcl W = \{S\cap U : S\in\mcl S_\theta^\ep\}$. 
\end{proof}

\begin{proof}[Proof of Theorem~\ref{thm-msrble-general}] 
It will be convenient to only have to consider a finite set of squares in $\mcl S_\theta^\ep$, so we fix a large bounded, connected open set $V\subset U$ (if $U$ itself is bounded, we can just take $V= U$). Let 
\eqbn
\mcl S_\theta^\ep(V) := \left\{S\in\mcl S_\theta^\ep : S\cap V \not=\emptyset \right\} .
\eqen 
Also fix points $z,w\in V$. We will show that the internal distance $D(z,w ; V)$ is a.s.\ determined by $h$.
Letting $z,w$ vary over $V\cap \BB Q^2$ and then letting $V$ increase to all of $U$ will conclude the proof. 
\medskip

\noindent\textit{Step 1: application of the Efron-Stein inequality.}
By Lemma~\ref{lem-internal-msrble}, $D$ is a.s.\ given by a measurable function of $h,\theta$, and the set of internal metrics $\{ D(\cdot,\cdot ; S\cap U) : S\in \mcl S_\theta^\ep\}$. 
By Lemma~\ref{lem-square-ind}, these internal metrics are conditionally independent given $(h,\theta)$. 
Hence, for each $S\in\mcl S_\theta^\ep(V)$, we can produce a new random metric $D^S$ by re-sampling $D (\cdot,\cdot;S \cap U)$ from its conditional law given $(h,\theta)$ and leaving $D (\cdot,\cdot ; S')$ unchanged for each $S' \in \mcl S_\theta^\ep\setminus \{S\}$. 
This metric satisfies $(h,\theta,D  ) \eqD (h,\theta,D^S)$. 

By the Efron-Stein inequality~\eqref{eqn-efron-stein}, applied under the conditional law given $(h,\theta)$, a.s.\ 
\eqb \label{eqn-msrble-efron-stein}
\op{Var}\left[ D (z,w ; V) \,|\, h , \theta \right] 
\leq \frac12 \sum_{S\in\mcl S_\theta^\ep(V)} \BB E\left[ \left( D^S(z,w ; V) - D(z,w ; V)    \right)^2 \,|\, h   , \theta \right]  .
\eqe 
Since the conditional laws of $  ( D , D^S)$ and $(D^S , D)$ given $(h,\theta)$ agree, the conditional law of $ D(z,w ; V) - D^S(z,w ; V)$ is symmetric around the origin, so each summand in~\eqref{eqn-msrble-efron-stein} satisfies
\eqb\label{eqn-es-sym}
\BB E\left[ \left(  D^S(z,w;  V) - D(z,w ; V) \right)^2 \,|\, h   , \theta \right] =  2 \BB E\left[ \left(  D^S(z,w ; V) - D(z,w ; V) \right)_+^2 \,|\, h   , \theta \right] ,
\eqe
where $(x)_+ = x$ if $x\geq 0$ or 0 if $x\leq 0$. 
Most of the rest of the proof is devoted to showing that the right side of~\eqref{eqn-msrble-efron-stein} tends to zero a.s.\ as $\ep\rta 0$.  
\medskip

\noindent\textit{Step 2: comparison of $D$ and $D^S$.}
Since $(h, D^S) \eqD (h, D)$, Lemma~\ref{lem-cond-bded} implies that if $u,v \in U$, then a.s.\ 
\eqbn
D(u,v ;V) , D^S(u,v;V) \in \left[ C^{-1 } \BB E[D(z,w;V)\,|\,h] ,  C \BB E[D(z,w;V)\,|\,h] \right] .
\eqen
This holds a.s.\ for all $u,v \in \ol U\cap \BB Q^2$ simultaneously, so since $D(\cdot,\cdot ;V)$ and $D^S(\cdot,\cdot ;V)$ are continuous metrics on $V$, a.s.\ 
\eqb \label{eqn-square-bilip}
C^{-2} D(u,v ; V) \leq D^S(u,v; V) \leq C^2 D(u,v ; V) ,\quad\forall u,v \in V .
\eqe
 
Since $D(\cdot,\cdot ; V)$ is a length metric, we can choose, in a manner depending only on $(h,D)$, a path $P : [0, |P| ] \rta V$ from $z$ to $w$ in $V$ whose $D$-length is at most $D(z,w;V) + \ep^2$. 
Henceforth fix such a path and recall~\eqref{eqn-geodesic-grid}. 
By the definition of $D^S$,
\eqb \label{eqn-grid-path-typical}
 \op{len}\left( P\cap S' ;D  \right) =  \op{len}\left( P\cap S'  ; D^S \right) ,\quad\forall S'\in\mcl S_\theta^\ep(V) \setminus \{S\} .
\eqe  
By~\eqref{eqn-square-bilip}, 
\eqb \label{eqn-grid-path-main}
C^{-2} \op{len}\left( P \cap S ; D \right)   \leq  \op{len}\left( P \cap S  ; D^S \right) \leq C^2 \op{len}\left( P \cap S ; D  \right)  .
\eqe  
By Lemma~\ref{lem-random-grid}, a.s.\ the contribution to the $D$-length of $P$ of the intersections of $P$ with the boundaries of the squares in $\mcl S_\theta^\ep$ is zero.
Since $D$ and $D^S$ are a.s.\ bi-Lipschitz equivalent (by~\eqref{eqn-square-bilip}), an argument as in the proof of Lemma~\ref{lem-internal-msrble} shows that the same is true with the $D^S$-length in place of the $D$-length. 
By combining this with~\eqref{eqn-grid-path-typical} and~\eqref{eqn-grid-path-main}, we get that a.s.\
\eqbn
D^S(z,w;V) \leq \op{len}(P  ; D^S ) = \sum_{S' \in \mcl S_\theta^\ep(V)} \op{len}\left( P\cap S' ; D^S \right) \leq D(z,w;V) + C^2 \op{len}\left( P \cap S ; D \right) + \ep^2 .
\eqen 
Therefore, a.s.,
\eqb \label{eqn-es-metric-diff}
\left( D^S(z,w;V)  - D(z,w;V) \right)_+ \leq C^2 \op{len}\left( P \cap S ; D\right) + \ep^2   .
\eqe 
By plugging~\eqref{eqn-es-metric-diff} into~\eqref{eqn-es-sym} and then into~\eqref{eqn-msrble-efron-stein} and then applying the Cauchy-Schwarz inequality,  
\allb \label{eqn-es-sum}
 &\op{Var}\left[ D(z,w; V) \,|\, h , \theta \right] \notag \\
 \leq&   \BB E\left[ \sum_{S\in\mcl S_\theta^\ep(V)} \left( C^2 \op{len}\left( P \cap S ; D \right)+\ep^2 \right)^2 \,|\, h , \theta \right]   \notag \\
\leq&  2C^4 \BB E\left[ \sum_{S\in\mcl S_\theta^\ep(V)} \left( \op{len}\left( P \cap S ; D \right) \right)^2 \,|\, h , \theta \right] +  2 \ep^4 \# \mcl S_\theta^\ep(V) \notag \\
\leq&  2 C^4 \BB E\left[\sum_{S\in\mcl S_\theta^\ep(V)} \op{len}\left( P \cap S ; D \right)   \left( \max_{S\in\mcl S_\theta^\ep(V)}  \op{len}\left( P \cap S ; D  \right)  \right)  \,|\, h , \theta \right] +  2 \ep^4 \# \mcl S_\theta^\ep(V)  \notag \\
\leq&  2 C^4 \BB E\left[ D (z,w;V)^2    \,|\, h , \theta \right]^{1/2}  \BB E\left[ \left( \max_{S\in\mcl S_\theta^\ep(V)}  \op{len}\left( P \cap S ; D \right)  \right)^2 \,|\, h , \theta \right]^{1/2}  +   2 \ep^4 \# \mcl S_\theta^\ep(V)  .
\alle
\medskip

\noindent\textit{Step 3: conclusion.}
We will now argue that the right side of~\eqref{eqn-es-sum} tends to zero a.s.\ as $\ep \rta 0$. 
Since $V$ is bounded, we have $\#\mcl S_\theta^\ep(V) = O_\ep(\ep^{-2})$, so $2\ep^4 \#\mcl S_\theta^\ep(V) \rta 0$ as $\ep\rta 0$. 
By Lemma~\ref{lem-cond-bded}, a.s.\ the first expectation the last line of~\eqref{eqn-es-sum} is finite. 
We will now argue that the second expectation a.s.\ tends to zero as $\ep\rta 0$. 
The path $P$ is a.s.\ contained in $V$, so in particular the range of $P$ is a compact subset of $U$. 

Since $\op{len}(P ; D) \leq D(z,w) + \ep^2$, for any $S\in\mcl S_\theta^\ep(V)$ the $D$-length of $P\cap S$ is at most $\sup_{u,v \in S} D (u,v) +\ep^2$ (otherwise, by replacing the segment of $P$ between the first and last points of $\ol S$ hit by $P$, we could find a path from $z$ to $w$ of $D$-length smaller than $D(z,w)$). 
Since $D$ is a continuous metric on $ U$ and the Euclidean side length of each $S\in\mcl S_\theta^\ep$ is $\ep$, it is a.s.\ the case that
\eqb \label{eqn-cond-small}
\lim_{\ep\rta 0}  \max_{S\in\mcl S_\theta^\ep(V)}  \op{len}\left( P \cap S ; D  \right) 
\leq \lim_{\ep\rta 0} \max_{S\in\mcl S_\theta^\ep(V) :  P\cap S \not=\emptyset } \left(  \sup_{u,v \in S} D (u,v)  + \ep^2 \right) 
= 0 .
\eqe
Each of the random variables $\op{len}\left( P \cap S ; D  \right)$ is bounded above by $\op{len}(P;D)$, which by Lemma~\ref{lem-cond-bded} and our choice of $P$ is a.s.\ bounded above by the $h$-measurable random variable $C\BB E[D(z,w;V)\,|\,h] + \ep^2$. By~\eqref{eqn-cond-small} and the bounded convergence theorem, the second expectation in the last line of~\eqref{eqn-es-sum} a.s.\ tends to zero as $\ep\rta 0$. 
Consequently, a.s.\ $\op{Var}\left[ D(z,w ; V) \,|\, h , \theta \right]   \rta 0$ as $\ep\rta 0$, so a.s.\ $D(z,w;V)$ is determined by $(h,\theta)$. 
Since $D(\cdot,\cdot; V)$ is continuous and this holds for any fixed choice of $z,w \in V$, a.s.\ $D(\cdot,\cdot;V)$ is determined by $(h,\theta)$. 
Since $(h,D)$ is independent from $\theta$, a.s.\ $D(\cdot,\cdot ;V)$ is determined by $h$. 
Since $U$ is a length metric, letting $V$ increase to all of $U$ shows that a.s.\ $D$ is determined by~$h$. 
\end{proof}

\bibliography{cibiblong,cibib}
\bibliographystyle{hmralphaabbrv}

\end{document}